\documentclass[reqno]{amsart}
\usepackage[utf8]{inputenc}
\usepackage{amsmath, amsthm, amssymb}
\usepackage{color}
\usepackage{tikz} 
\usepackage{afterpage}
\usepackage{float}
\usepackage{caption, subcaption} 
\usepackage{enumitem} 
\usepackage{longtable} 
\usepackage[section]{placeins}
\usepackage{comment} 
\usepackage{array} 
\usepackage{amsthm}
\usepackage{amsmath}
\usepackage{amsfonts}
\usepackage{amssymb}
\usepackage{enumitem}
\usepackage{mathtools}
\usepackage{relsize}
\usepackage{cite}
\usepackage[colorlinks=true,linkcolor=blue]{hyperref}
\usepackage{graphicx}
\usepackage{fancyhdr}
\usepackage{lipsum}
\usepackage{ytableau}
\usepackage{tabularx}
\usepackage{tikz}
\usepackage{cleveref}

\usetikzlibrary{shapes.geometric} 
\usetikzlibrary{matrix,decorations.pathreplacing,calc}

\pgfkeys{tikz/mymatrixenv/.style={decoration=brace,every left delimiter/.style={xshift=3pt},every right delimiter/.style={xshift=-3pt}}}
\pgfkeys{tikz/mymatrix/.style={matrix of math nodes,left delimiter=[,right delimiter={]},inner sep=2pt,column sep=1em,row sep=0.5em,nodes={inner sep=0pt}}}
\pgfkeys{tikz/mymatrixbrace/.style={decorate,thick}}

\newtheorem{theorem}{Theorem}
\numberwithin{theorem}{subsection}
\newtheorem{lemma}[theorem]{Lemma}

\newtheorem{proposition}[theorem]{Proposition}
\newtheorem{conjecture}[theorem]{Conjecture}
\newtheorem{corollary}[theorem]{Corollary}
\newtheorem{question}[theorem]{Question}
\theoremstyle{definition}
\newtheorem{definition}[theorem]{Definition}
\newtheorem{example}[theorem]{Example}
\newtheorem{notation}[theorem]{Notation}
\newtheorem{remark}[theorem]{Remark}

\DeclareMathOperator{\End}{End}

\DeclareMathOperator{\id}{id}
\DeclareMathOperator{\ev}{ev}
\DeclareMathOperator{\coev}{coev}

\DeclareMathOperator{\floor}{floor}
\DeclareMathOperator{\roof}{roof}

\DeclareMathOperator{\length}{length}
\DeclareMathOperator{\col}{col}

\DeclareMathOperator{\height}{height}
\DeclareMathOperator{\rowlength}{rowlength}

\DeclareMathOperator{\Rad}{Rad}
\DeclareMathOperator{\PGL}{PGL}
\DeclareMathOperator{\PSL}{PSL}
\DeclareMathOperator{\SL}{SL}

\newcommand{\mc}{\mathcal}

\newcommand{\kk}{\Bbbk}
\newcommand{\rep}{{\sf{rep}}}
\newcommand{\Ver}{{\sf{Ver}}}

\title[Non-negligible summands in tensor powers of representations of $p$-groups]{Non-negligible summands in tensor powers of some modular representations of finite $p$-groups}

\author{Kent B. Vashaw}
\address{
Department of Mathematics\\
UCLA\\
Los Angeles, CA 90095\\
U.S.A.}
\email{kentvashaw@math.ucla.edu}

\author{Justin Zhang}
\address{Department of Mathematics\\
MIT\\
Cambridge, MA 02139\\
U.S.A.}
\email{juszha@mit.edu}

\date{}

\dedicatory{Dedicated to Dan Nakano on the occasion of his 60th birthday, with admiration}

\begin{document}


\subjclass{
20C20, 
18M05
}

\begin{abstract}
Let $p>0$ be a prime, $G$ be a finite $p$-group and $\Bbbk$ be an algebraically closed field of characteristic $p$. Dave Benson has conjectured that if $p=2$ and $V$ is an odd-dimensional indecomposable representation of $G$ then all summands of the tensor product $V \otimes V^*$ except for $\Bbbk$ have even dimension. It is known that the analogous result for general $p$ is false. In this paper, we investigate the class of graded representations $V$ which have dimension coprime to $p$ and for which $V \otimes V^*$ has a non-trivial summand of dimension coprime to $p$, for a graded group scheme closely related to $\mathbb{Z}/p^r \mathbb{Z} \times \mathbb{Z}/p^s \mathbb{Z}$, where $r$ and $s$ are nonnegative integers and $p>2$. We produce an infinite family of such representations in characteristic 3 and show in particular that the tensor subcategory generated by any of these representations in the semisimplification contains the modulo $3$ reduction of the category of representations of the symmetric group $S_3$. Our results are compatible with a general version of Benson's conjecture due to Etingof.

\end{abstract}

\maketitle

\section{Introduction}

In modular representation theory of finite groups, that is, the study of representations of a finite group $G$ whose order is divisible by a prime $p$ over a field of characteristic $p$, many seemingly easy-to-state questions about the decompositions of tensor products remain unsolved. One such fundamental question that has recently generated research interest is the following.

\begin{question}
\label{mainq}
For what finite-dimensional indecomposable $G$-representations $V$ does the tensor product $V\otimes V^*$ break down into a direct sum of the trivial representation $\kk$ and indecomposable representations of dimension divisible by $p$?
\end{question}

By a theorem of Benson--Carlson \cite[Theorem 2.1]{Benson-Carlson}, the trivial representation $\kk$ appears as a direct summand of $V \otimes W$, with $V$ and $W$ indecomposable, if and only if $W \cong V^*$ and the dimension of $V$ is not divisible by $p$. In this case the multiplicity of $\kk$ in the decomposition of $V \otimes V^*$ is 1. We adopt the terminology of \cite{BENSON202024}: when $V$ has dimension non-divisible by $p$, we say that $V$ is a {\emph{$p'$-representation}} or {\emph{non-negligible}}, and when $V \otimes V^*$ breaks down into a direct sum of $\kk$ and representations of dimension divisible by $p$ we call $V$ {\em{$p'$-invertible.}} Using this terminology, Question \ref{mainq} can now be restated as: which $p'$-representations are $p'$-invertible?

This question can be restated in yet another way in the language of semisimplifications of tensor categories. The semisimplification of a spherical tensor category (see \cite{Etingof-Ostrik}) is a semisimple tensor category, where simple objects correspond to those indecomposable objects of the original category whose dimension (considered as an element of $\kk$) is nonzero (see Theorem \ref{semis-recall}). We write $\rep(\kk,G)$, or just $\rep(G)$ when the field $\kk$ is understood, for the category of finite-dimensional representations of $G$ over $\kk$. The category $\rep(G)$ is a symmetric tensor category, and has a semisimplification $\overline{\rep}(G)$. Therefore, Question \ref{mainq} can be restated as: what are the tensor-invertible representations of the semisimplification of $G$-representations $\overline{\rep}(G)$? Since semisimplifications of categories of representations for finite groups have recently been an important tool for studying general symmetric tensor categories \cite{Coulembier-Etingof-Ostrik, Benson-Etingof-Ostrik}, this question is pressing. 

Based on extensive evidence collected using computer algebra systems, Dave Benson has made a striking conjecture for the answer to Question \ref{mainq} in characteristic 2, see \cite[Conjecture 1.1]{BENSON202024}:

\begin{conjecture}[Benson's Conjecture]
\label{Bensonconj}
If $G$ is a finite 2-group, $\kk$ a field of characteristic $2$, and $V$ an odd-dimensional representation of $G$, then all indecomposable summands of $V \otimes V^*$ have dimension divisible by 2 except for the single summand isomorphic to $\kk$. That is, every $2'$-representation is $2'$-invertible. 
\end{conjecture}

In fact, this is a weak version of Benson's Conjecture; he also gives a strong version of the conjecture, which is that all summands of $V \otimes V^*$ have dimension divisible by 4, based on ample computational evidence.

One useful consequence of \Cref{Bensonconj} would be that all tensor powers of a $2'$-representation $V$ would have a unique $2'$-summand. In this case, Benson has further conjectures on the growth of the dimension of these $2'$-summands \cite[Conjecture 1.0.2]{cao2023decomposition}, which has been verified in a few small examples \cite[Theorem 1.0.3]{cao2023decomposition}. Alternative closely-related growth functions related to either the number of indecomposable summands in a tensor power or the dimensions of non-projective summands in tensor powers have also received significant recent attention, see e.g.~ \cite{Benson-Symonds,Upadhyay2021, Upadhyay2022,Chirvasitu-Hudson-Upadhyay,Lacabanne-Tubbenhauer-Vaz-2023,Coulembier-Ostrik-Tubbenhauer,Coulembier-Etingof-Ostrik-pre, CEOT}.

For $p>2$, however, there is not even a conjectural description of the $p'$-invertible representations. The naive extension of Benson's Conjecture -- that all $p'$-representations are $p'$-invertible -- fails in even very low-dimensional examples starting with $p=3$. Etingof has proposed the following generalization of Benson's conjecture for all $p$ based on the semisimplification $\overline{\rep}(G)$.

\begin{conjecture}[Generalized Benson Conjecture]
    \label{etingof-conj}
    Let $G$ be a $p$-group for $p$ a prime integer, $\kk$ a field of characteristic $p$, and $\mc{C}$ a tensor subcategory of $\overline{\rep}(G)$ which is equivalent to $\rep(G')$ for some finite group $G'$ whose order is coprime to $p$. Then every prime dividing the order of $G'$ is less than $p$. 
\end{conjecture}

We clarify precisely what we mean by ``equivalent" when we restate this conjecture in the language of induced $p'$-symmetries below in Conjecture \ref{etingof-conj-2}.

More generally, Conjecture \ref{etingof-conj} suggests the following question.

\begin{question}
    \label{which-subcats}
    For $G$ and $p$ as above, what tensor categories appear as tensor subcategories of $\overline{\rep}(G)$?
\end{question}

In this paper, we study \Cref{mainq} and \Cref{which-subcats} for certain representations of a finite group scheme $\alpha_p(r,s)$ which is closely related to $\mathbb{Z}/p^r \mathbb{Z} \times \mathbb{Z} / p^s \mathbb{Z}$. This group scheme is $\mathbb{Z}^2$-graded. We will consider cyclic representations which are generated in degree $(0,0)$. 

In our first result, we prove a general theorem. Here we will use the correspondence between cyclic representations $V$ and diagrams with rows and columns as in Section \ref{background}. 

\begin{theorem}
    \label{maintheorem}
    Let $p$ be a prime integer, $h$ be a positive integer, and $r$ and $s$ be nonnegative integers. Let $V_h$ be the cyclic representation of dimension $h$ generated in degree $(0,0)$ for $\alpha_p(r,s)$ such that $x$ acts by $0$. Suppose $W$ is a direct summand of $V_h \otimes V_h^*$ whose diagram is a column, generated in degree $(0,-n)$ and of dimension $2n+1$ or $2n+2$. Set $m_1:=\lceil \log_p(h) \rceil$ and $m_2:=\lceil \log_p(\dim(W)) \rceil$, that is, $m_1$ is the smallest integer such that $h \leq p^{m_1}$ and $m_2$ is the smallest integer such that $\dim(W) \leq p^{m_2}$. Set $g:=\max(m_1, m_2)$. Suppose $V$ is a cyclic representation such that every column has length equal to 0 or $h$ modulo $p^g$ and $\dim(V)$ is coprime to $p$. Then $W$ is also a direct summand of $V \otimes V^*$. 
\end{theorem}

In particular, we apply this theorem in the case $p=3$ to produce an infinite family of $p'$-representations which are not $p'$-invertible.

\begin{theorem}
    \label{mainthm-p3}
    Let $p=3$ and let $r$ and $s$ be nonnegative integers. Let $V$ be a cyclic representation for $\alpha_p(r,s)$ of dimension coprime to $p$ such that the diagram of $V$ satisfies either 
    \begin{enumerate}
        \item every column has length equal to 0 or $5$ modulo $9$;
        \item every row has length equal to 0 or $5$ modulo $9$.
    \end{enumerate}
    Then $V$ is not $p'$-invertible, and the tensor subcategory of $\overline{\rep}(\alpha_p(r,s))$ generated by $\overline{V}$ contains a semisimple tensor subcategory $\mc{C}$ which is equivalent to $\rep(\mathbb{Z}/2\mathbb{Z} \ltimes \mu_3)$, where $\mu_3$ is the dual group scheme to $\mathbb{Z}/3\mathbb{Z}$.
\end{theorem}

In summary, this allows us to identify a particular nontrivial tensor subcategory of $\overline{\rep}(\alpha_3(r,s))$ (motivated by Question \ref{which-subcats}), and we do indeed find that it is in line with Conjecture \ref{etingof-conj}, that is, it is not equivalent to $\rep(G')$ for $G'$ a group with order divisible by a prime larger than $p$.

The paper is organized as follows. In Section \ref{background}, we set notation regarding the group scheme $\alpha_p(r,s)$ and its representations, recall the semisimplification construction, and give more background on Etingof's generalization of Benson's conjecture. In Section \ref{proof-mainthms} we prove the main theorems of the paper. We first characterize certain maps between tensor products of cyclic representations by a system of equations. Using this characterization, we are able to identify direct summands of $V \otimes V^*$ for certain representations $V$. Lastly, we consider the consequences in the case $p=3$, identifying a particular nontrivial tensor subcategory that appears in $\overline{\rep}(\alpha_3(r,s))$.

\section*{Acknowledgements}
We would like to thank Pavel Etingof for suggesting this project, for many helpful discussions, and for comments on an earlier draft of this paper. We also thank Dave Benson for helpful conversations and George Cao for sharing code he had written with us. Computations done in the Magma computer algebra system \cite{BCP1997} were essential to finding the results in this paper. This research was conducted as part of the MIT PRIMES-USA program, and we thank PRIMES-USA for making this project possible. Research of K.~B.~V.~was partially supported by NSF Postdoctoral Fellowship DMS-2103272 and by an AMS-Simons Travel Grant.

\section{Background}
\label{background}
\subsection{Representations for the group scheme $\alpha_p(r,s)$}

Let $\kk$ denote an algebraically closed field of characteristic $p$ for some prime $p>0$. We denote $G := \mathbb{Z}/p^r \mathbb{Z} \times \mathbb{Z} / p^s \mathbb{Z}.$ The generating set of $G$ is given by $\{g,h\}$, where $g^{p^r}=h^{p^s}=1$ and $gh=hg$. Define elements $x := g-1$ and $y := h-1$ in the group algebra $\kk G$. Note that $x$ and $y$ are both nilpotent because 
$$x^{p^r} = (g-1)^{p^{r}} = g^{p^r}-1 = 1-1=0,$$
and the analogous relation holds for $y$. Thus, to define a representation for $\kk G$, it suffices to define the actions of $x$ and $y$: these actions should commute and be nilpotent, of orders $p^r$ and $p^s$, respectively.

We define the Hopf algebra $\kk \alpha_p(r,s)$ in the following way. As an algebra, let $\kk \alpha_p(r,s)$ be the group algebra of $\mathbb{Z}/p^r \mathbb{Z}\times \mathbb{Z}/p^s \mathbb{Z}$; endow $\kk \alpha_p(r,s)$ with a comultiplication given by 
\[
x \mapsto x \otimes 1 + 1 \otimes x, \;\; y \mapsto y \otimes 1 + 1 \otimes y.
\]
This alternative definition of comultiplication provides $\kk \alpha_p(r,s)$ with the structure of a finite-dimensional cocommutative Hopf algebra, i.e.~ the dual of the coordinate ring of a finite group scheme (see \cite[Chapter 1]{Waterhouse1979}). The representations of $\kk \alpha_p(r,s)$ are canonically identified with the representations of this group scheme; we denote the group scheme by $\alpha_p(r,s)$. Note that while 
\[
\rep(\alpha_p(r,s)) \cong \rep(\kk G)
\]
as abelian categories, they are not equivalent as tensor categories, since the comultiplication for $\kk \alpha_p(r,s)$ and $\kk G$ differ. 

To define a $\mathbb{Z}^2$ grading on $\kk \alpha_p(r,s)$, we set the degree of $x$ to be $(1,0)$ and the degree of $y$ to be $(0,1)$, with $x$ and $y$ as elements of $\kk \alpha_p(r,s)$ as before. This grading gives $\kk \alpha_p(r,s)$ the structure of a graded Hopf algebra, and so we can consider graded representations, to which we associate graded diagrams.

To formally define these graded diagrams consistent with the notation used in \cite{cao2023decomposition}, consider a partition $\lambda = (\lambda_1, \ldots, \lambda_n)$ for which $\lambda_1\geq \lambda_2\geq \cdots \geq \lambda_n > 0$. The representation corresponding to such a partition that we consider has basis elements $v_{i,j}$ such that $0\leq i\leq n-1$ and $0\leq j\leq \lambda_{i+1}-1$. In such a representation, the $x$ action applied to the basis element $v_{i,j}$ yields $v_{i+1,j},$ and similarly the $y$ action applied to $v_{i,j}$ yields $v_{i,j+1}$. If $v_{i+1,j}$ or $v_{i,j+1}$ do not exist in these respective cases, we say that the basis element $v_{i,j}$ is sent to $0$.

Note that for this to define a valid representation, there must be at most $p^s$ values of $j$ for each $i$ such that $v_{i,j}$ exists, and similarly, at most $p^r$ values of $i$ for each $j$ for which $v_{i,j}$ exists. We can further pictorially depict the grading of these representations with graded diagrams. For all $i$ and $j$ for which $v_{i,j}$ is a basis element, we draw in the grid box, or cell, in degree $(i,j)$. We write a 1 in each box to indicate that there is one basis vector in that given degree. Colloquially, when looking at a diagram, the action of $x$ is ``move to the right", and the action of $y$ is ``move up".

\begin{example}
\label{monom-firstex}
For $G := \mathbb{Z}/5\mathbb{Z}\times \mathbb{Z}/25\mathbb{Z},$ the graded diagram corresponding to the partition $\lambda=(6,3,2,2)$ is given by 
\vspace{\baselineskip}
\begin{center}
\ytableausetup{notabloids}
\begin{ytableau}
1 \\
1 \\
1 \\
1 & 1\\
1 & 1 & 1 & 1\\
1 & 1 & 1 & 1\\
\end{ytableau}
\end{center}
\vspace{\baselineskip}
As can be seen from the diagram, there are basis vectors in $v_{i,j}$ for $(i,j)\in S$ for 
\[S := \{(0,0), (0,1), (0,2), (0,3), (0,4), (0,5), (1,0), (1,1), (1,2), (2,0), (2,1), (3,0), (3,1)\}.\]
We have, for instance, $x.v_{0,0} = v_{1,0}$, and $x.v_{3,0} = x.v_{1,2} = y.v_{0,5} = y.v_{1,2}=0$.
\end{example}

Throughout this paper, we often refer to the \textit{rows} and \textit{columns} within the graded diagrams corresponding to the representations mentioned above. Row $a$ refers to the group of cells corresponding to $v_{i,j}$ with $j = a.$ Similarly, column $b$ refers to the group of cells corresponding to $v_{i,j}$ for which $i=b$.

Furthermore, if $\lambda=(\lambda_1,...,\lambda_n)$ is a partition such that $\lambda_1 \geq \lambda_2 \geq... \geq \lambda_n$ and $\mu=(\mu_1,..., \mu_n)$ is another partition such that $\mu_1 \geq \mu_2 \geq... \geq \mu_n$ and $\mu_i < \lambda_i$ for all $i$, we can form a skew diagram consisting where we remove the diagram of $\mu$ from the diagram of $\lambda$. We will denote this skew partition by $\lambda\backslash \mu$. As long as there are at most $p^s$ cells remaining in each column (after having removed $\mu$ from $\lambda$) and there are at most $p^r$ cells remaining in each row (again, after having removed $\mu$ from $\lambda$), then the skew diagram corresponds to a representation for $\alpha_p(r,s)$: each cell corresponds to a basis element, the action of $x$ corresponds to moving right, and the action of $y$ corresponds to moving up. 

\begin{example}
\label{monom-secondex}
Suppose $\lambda=(6,3,2,2)$ and $\mu=(2,1,1,0)$. The skew diagram given by
\vspace{\baselineskip}
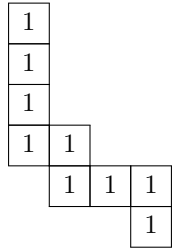
\begin{center}
\ytableausetup{notabloids}
\begin{ytableau}
1 \\
1 \\
1 \\
1 & 1\\
\none & 1 & 1 & 1\\
\none & \none & \none & 1\\
\end{ytableau}
\end{center}
\vspace{\baselineskip}
corresponds to a representation $V$ for $\alpha_p(r,s)$ if $4 \leq p^s$ and $3 \leq p^r$. If it exists, this representation $V$ is 9-dimensional.
\end{example}

We now verify that when $V$ corresponds to a connected (skew) diagram, then $V$ is indecomposable. Although this is not difficult to observe, we prove a more general result for the sake of future application.

\begin{proposition}
\label{gr-group-comp}
Let $G$ be a finite group scheme with a $\mathbb{Z}^n$-grading. Let $V$ be a finite-dimensional graded representation for $G$. Then there exists a decomposition $V \cong V_1 \oplus...\oplus V_m$ of $V$ into indecomposable $G$-representations $V_1,\cdots,V_m$ in a way which respects the grading on $V$.
\end{proposition}

\begin{proof}
Write $\kk G:=\mc{O}(G)^*$ for the dual of the coordinate ring of $G$, that is, $\kk G$ is a cocommutative Hopf algebra such that $G$-representations are in natural bijection with $\kk G$-representations. We know that $\kk G$ is graded by $L:=\mathbb{Z}^n$ by assumption and finite-dimensional. Gradings by $L$ correspond to coactions by the group algebra $\kk L$, and $\kk L$ is a commutative and cocommutative Hopf algebra. The algebraic torus $T:=(\kk^{\times})^n$ is the group scheme which corresponds to $\kk L,$ that is, $\mc{O}(T) \cong \kk L$, see \cite[Example 11.15]{Lorenz2018}; and so there is a rational torus action of $T$ on $\kk G$. Likewise, we have an action of $T$ on $V$, since $V$ is assumed to be graded. Explicitly, the component $V_{\lambda}$ of $V$ graded by $\lambda:=(a_1,..., a_n) \in \mathbb{Z}^n$ is the $\lambda$ weight space of $T$, that is, $\alpha:=(t_1,...,t_n)$ acts on $v \in V_{\lambda}$ by 
\[
\alpha.v=\prod_{i=1}^n t_i^{a_i} v.
\]

Since the action of $G$ on $V$ is compatible with the $T$-action (i.e., the $L$-grading), $T$ acts naturally on $\End_G(V)$. By \cite{Kelarev1992}, the radical of $\End_G(V)$ is an $L$-homogeneous ideal, and so $T$ acts naturally on $A:=\End_G(V)/\Rad(\End_G(V))$, which is isomorphic to a product
\[
A=\End_G(V) / \Rad(\End_G(V)) \cong M_{n_1}(\kk) \times M_{n_2}(\kk) \times... \times M_{n_m}(\kk)
\]
of matrix algebras \cite[Theorem 3.5.4]{Etingof2011-ym}. Recall that the matrix blocks $M_{n_i}(\kk)$ correspond to the pairwise non-isomorphic indecomposable summands of $V$. Since the torus $T$ is connected \cite[Exercise 10.7]{MalleTesterman2011}, the action of $T$ on $A$ fixes each of the matrix algebra factors, that is, it restricts to an action of $T$ on $M_{n_i}(\kk)$ for each $i$. Recall that an action of $T$ on $M_{n_i}(\kk)$ corresponds to a map of algebraic groups
\[
T \to \PGL_{n_i}(\kk),
\]
see \cite[Exercise 10.17]{MalleTesterman2011}. Since we are assuming $\kk$ is algebraically closed, 
\[
\PGL_{n_i}(\kk) \cong \PSL_{n_i}(\kk) :=\SL_{n_i}(\kk) / \mu_{n_i},
\]
where $\mu_{n_i}$ is the group of $n_i$-roots of unity in $\kk$ embedded as scalar matrices in $\SL_{n_i}(\kk)$, see \cite[Example 5.49]{Milne}. Since there are no maps $T \to \mu_{n_i}$, the map $T\to \PGL_{n_i}(\kk)$ lifts to a map $T \to \SL_{n_i}(\kk)$. Since all maximal tori in $\SL_{n_i}(\kk)$ are conjugate (see \cite[Theorem 4.4(b)]{MalleTesterman2011}), we can choose a basis for $\kk^{n_i}$ such that the image of $T$ in $\SL_{n_i}(\kk)$ consists of diagonal matrices. Therefore, we can choose a complete system of orthogonal idempotents for $V$ which are preserved by the action of $T$; that is, we can decompose $V$ into indecomposable summands which are also $L$-graded, and the decomposition respects the grading. 
\end{proof}

\begin{corollary}
\label{indecomp-gr}
    If $V$ is an $\alpha_p(r,s)$-representation corresponding to a partition $\lambda \in \mathbb{Z}^n$, then $V$ is indecomposable. More generally, if $V$ is an $\alpha_p(r,s)$-representation corresponding to a skew partition $\lambda \backslash \mu$, then $V$ is indecomposable if and only if the diagram of $V$ is connected.
\end{corollary}

\begin{proof}
    It is straightforward to see that $V$ is indecomposable as a graded $\alpha_p(r,s)$-representation if and only if its diagram is connected. By Proposition \ref{gr-group-comp}, if $V$ is indecomposable as a graded $\alpha_p(r,s)$-representation, then it is indecomposable in $\rep(\alpha_p(r,s))$. 
\end{proof}

\begin{example}
    The representations corresponding to the diagrams given in Example \ref{monom-firstex} and \ref{monom-secondex} are both indecomposable in $\rep(\alpha_p(r,s))$.
\end{example}

\subsection{Induced $p'$-symmetries and the generalized Benson conjecture}

Let $G$ be a group whose order is coprime to $p$; in this case we call $G$ a $p'$-group. Recall that a $G$-representation whose dimension is divisible by $p$ is called {\it{negligible}}, a $G$-representation is called {\it{non-negligible}} if it is not negligible. A special case of non-negligible representations are the {\it{endotrivial}} representations, originally introduced by Dade in \cite{Dade1978}, which are those representations $V$ such that $V \otimes V^*$ decomposes into $\kk$ direct sum with a projective representation. The endotrivial representations have attracted significant focus in modular representation theory in recent years, see e.g.~ \cite{Puig1990,Carlson2006,CMN2006,CMN2009,CarlsonNakano2011}. Every endotrivial representation is $p'$-invertible, that is, if $V$ is endotrivial then every indecomposable summand of $V \otimes V^*$ other than $\kk$ has dimension divisible by $p$. 

To state the generalized Benson conjecture, we recall the semisimplification of a tensor category, originally due to Barrett and Westbury \cite{BarrettWestbury1999}, studied extensively by Etingof--Ostrik in \cite{Etingof-Ostrik}. 

\begin{theorem}
\label{semis-recall}
Let $\mathcal{C}$ be a spherical tensor category. Then there exists a semisimple tensor category $\overline{\mathcal{C}}$, called the {\emph{semisimplification of $\mathcal{C}$}}, where:
\begin{enumerate}
    \item objects of $\mathcal{C}$ and $\overline{\mathcal{C}}$ are in bijection, and there is a canonical tensor functor $\mathcal{C} \to \overline{\mathcal{C}}$;
    \item the simple objects of $\overline{\mathcal{C}}$ correspond precisely to the indecomposable non-negligible objects of $\mathcal{C}$.
\end{enumerate}
\end{theorem}
For background on tensor categories (and in particular spherical and symmetric tensor categories), see \cite{EGNO2015, EtingofKannan2023}. The category $\rep(G)$ of representations of a finite group is indeed a spherical tensor category, so it has a semisimplification $\overline{\rep}(G)$. If $V$ is a representation of $G$, we denote the corresponding object of $\overline{\rep}(G)$ by $\overline{V}$. 

We will say that a tensor category $\mc{C}$ is equivalent to $\rep(G)$ {\it{up to twisting the symmetric structure}} if $\mc{C} \cong \rep(G,z)$ for some central $z \in G$ with $z^2=1$ as in \cite[Example 9.9.1(3)]{EGNO2015}. Since we will not use the symmetric structure explicitly, we omit recalling the details.

\begin{definition}
    \label{induced-symm}
    Let $G$ be a $p$-group. Let $V$ be an indecomposable representation, and consider the full subcategory $\mc{D}$ of $\overline{\rep}(G)$ generated, as a tensor category, by the image of $V$. If $G'$ is a $p'$-group such that $\rep(G')$ is tensor-equivalent up to twisting the symmetric structure to a full subcategory $\mc{C}$ of $\mathcal{D}$, then we say that $(\mathcal{C},G')$ is an {\it{induced $p'$-symmetry of $V$}}. 
\end{definition}

To form the full tensor subcategory generated by the image of $V$, one iteratively takes dual objects, direct sums, direct summands, and tensor products.

\begin{remark}
    Suppose $\mathcal{C}$ is any full tensor subcategory of $\overline{\rep}(G)$. By \cite{Coulembier-Etingof-Ostrik}, $\mathcal{C}$ admits a fiber functor to $\Ver_p:=\overline{\rep}(\mathbb{Z}/p\mathbb{Z})$ (the universal Verlinde category introduced in \cite{Ostrik2020}), and hence by Tannakian reconstruction (see \cite[Section 1.7]{EtingofKannan2023}) it can be written $\mathcal{C} \cong \rep_{\Ver_{p}}(H)$ for some group scheme $H$ in $\Ver_p$. Note that it is not necessarily true that there is an actual finite group whose representation category is equivalent to $\mc{C}$, although if $p=2$ then it can indeed be written as representations of a group scheme, and if $p=3$ then it can be written as representations of a supergroup scheme.
\end{remark}

\begin{example}
    \label{pprime-inv}

    Let $V$ be a $p'$-invertible representation of $G$; by definition, $\overline{V} \otimes \overline{V}^* \cong \overline{\kk}$, i.e.~ $\overline{V}$ and $\overline{V^*}$ are tensor-invertible in $\overline{\rep}(G)$. Therefore, for any positive integer $n$, $\overline{V}^{\otimes n}$ and $(\overline{V^*})^{\otimes n}$ are also invertible, and so $\overline{V}^{\otimes n}$ and $(\overline{V}^*)^{\otimes n}$ are simple by \cite[Exercise 4.3.11(3)]{EGNO2015}. Therefore, any subcategory of the tensor subcategory generated by $\overline{V}$ is generated by the object $\overline{V}^{\otimes n}$ for some integer $n$, and simply consists of direct sums of the tensor powers of that object and its dual. If $V$ is self-dual, then the subcategory generated by $\overline{V}$ consists of direct sums of itself and $\overline{\kk}$. Hence the two simple objects of the subcategory generated by $\overline{V}$ would both be dimension 1. Here (and below) by dimension we mean Frobenius--Perron dimension, see \cite[Section 4.5]{EGNO2015}. We are using the fact that invertible objects all have dimension 1 \cite[Exercise 4.5.9]{EGNO2015} and that Frobenius--Perron dimension corresponds to vector space dimension \cite[Exercise 5.13.8]{EGNO2015} for finite group schemes. Note that the semisimplification functor does not necessarily respect dimension, \cite[Section 1.2]{Ostrik2020}, that is, when we discuss the dimension of $\overline{V}$ in $\overline{\rep}(G)$, this dimension might not be the same as the vector space dimension of $V$. By the discussion above and since the sum of the squares of the dimensions of the irreducible representations of a group must add up to the order of the group \cite[Corollary 1.35]{Lorenz2018}, the only possible induced $p'$-symmetries would be for $G'=\mathbb{Z}/2\mathbb{Z}$ the cyclic group of order 2 or for $G'=0$. 
\end{example}

\begin{example}
\label{benson-pprime-symm}
Let $p=2$. Benson's conjecture (Conjecture \ref{Bensonconj}) would imply that for any $V$ indecomposable in $\rep(G)$, either (1) $\overline{V} \cong 0$ (if the dimension of $V$ is divisible by $p$) or (2) $V$ is $p'$-invertible. If $V$ is $p'$-invertible, then if $V$ is self-dual, we have no nontrivial induced $p'$-symmetries by Example \ref{pprime-inv}. Assuming Benson's tensor powers conjecture (see \cite[Conjecture 1.0.2]{cao2023decomposition}), if $V$ is not self-dual then there are polynomials $f_0(x),..., f_{m-1}(x)$ such that the dimension of the unique odd-dimensional summand of $V^{\otimes n}$ is $f_i(n)$, if $n$ is congruent to $i$ modulo $m$. If the polynomials $f_0, \cdots f_{m-1}$ are all at least degree 1, then each of the tensor powers $\overline{V}^{\otimes n}$ is simple, and they are all pairwise non-isomorphic (by dimension consideration). Hence every tensor subcategory of the subcategory generated by $\overline{V}$ would have infinitely many objects, and therefore admit no induced $p'$-symmetries. Cases where one of the polynomials $f_i$ is degree 0 happen when one has a representation $V$ such that $V^{\otimes n} \cong \kk$. Examples of this include the so-called torsion endotrivial modules; these were classified by Carlson and Th\'evenaz, see \cite{CarlsonThevenaz2005}. In substantial computer evidence collected by Dave Benson, any induced $p'$-symmetry for $p=2$ has $G' \cong \mathbb{Z}/2^n\mathbb{Z}$ for some $n$, see \cite[p.~ 230]{EtingofKannan2023}.
\end{example}

Based on Example \ref{benson-pprime-symm}, Etingof has proposed the generalization of Benson's conjecture (\Cref{etingof-conj}), which we now restate in the language of induced $p'$-symmetries. 

\begin{conjecture}[Generalized Benson conjecture]
\label{etingof-conj-2} Let $V$ be an indecomposable representation for a $p$-group $G$ over an algebraically closed field $\kk$ of characteristic $p$. If $(\mc{C}, G')$ is an induced $p'$-symmetry of $V$, then every prime dividing the order of $G'$ is less than $p$.
\end{conjecture}

In particular, to gather information related to Etingof's conjecture, recall that we have the broader question as posed in Question \ref{which-subcats}: What tensor categories appear as subcategories of $\overline{\rep}(G)$? We will consider these questions throughout this paper for $\alpha_p(r,s)$; although this is a group scheme rather than a finite group, Etingof has conjectured that the analogous statement to Conjecture \ref{etingof-conj-2} holds for $\alpha_p(r,s)$ as well.

In the remainder of this section, we discuss a few concrete examples of induced $p'$-symmetries and non-negligible representations (i.e. $p'$-representations) which are not $p'$-invertible. Generalizing some of these examples to infinite families of $p'$-representations which are not $p'$-invertible is done in the following section.

\begin{example}
    Suppose $G=\mathbb{Z}/p\mathbb{Z}$ is the cyclic group of order $p$. The semisimplification $\overline{\rep}(G)$ is the previously mentioned universal Verlinde category $\Ver_p$ introduced by Ostrik \cite{Ostrik2020}. This category has $p-1$ simple objects, denoted $\overline{L_1}, \cdots, \overline{L_{p-1}}$. For $p=2$, $\Ver_p$ is equivalent to the category of vector spaces, hence has no nontrivial induced $p'$-symmetries. For $p=3$, $\Ver_p$ is equivalent to the category of super vector spaces \cite[Example 3.2(ii)]{Ostrik2020}. For $p>3$, the only tensor subcategories of $\Ver_p$ are: a copy of the category of vector spaces (direct sums of the tensor unit $\overline{L_{1}}$), a copy of super-vector spaces (direct sums of $\overline{L_{1}}$ and $\overline{L_{p-1}}$), the category known as $\Ver_p^+$ (direct sums of the objects $\overline{L_1}, \overline{L_3}, \overline{L_5}\cdots, \overline{L_{p-2}}$), and the whole category $\Ver_p$ \cite[Proposition 3.3(ii)]{Ostrik2020}. Since the category of super-vector spaces is equivalent to $\rep(\mathbb{Z}/2\mathbb{Z})$ up to a twisting of the symmetric structure (assuming $p>2$) \cite[Example 9.9.3]{EGNO2015}, and since $\Ver_p^+$ is not tensor-equivalent to $\rep(G')$ for any finite group $G'$ \cite[Theorem 4.71(i)]{Benson-Etingof-Ostrik}, the only induced $p'$-symmetries for $\mathbb{Z}/p\mathbb{Z}$ are for $G'$ as the trivial group and as $\mathbb{Z}/2\mathbb{Z}$.
\end{example}

\begin{example}
\label{p3-v5-ex}
    Let $V_5$ be the cyclic 5-dimensional representation and $V_7$ be the cyclic 7-dimensional representation for $\alpha_3(0,2)$. That is, $V_5$ and $V_7$ correspond to the diagrams 
    \vspace{\baselineskip}
    \begin{center}
    \ytableausetup{notabloids}
\begin{ytableau}
1 \\
1 \\
1 \\
1\\
1 \\
\end{ytableau}, \qquad
\begin{ytableau}
1 \\
1 \\
1 \\
1\\
1 \\
1\\
1\\
\end{ytableau}
\end{center}
\vspace{\baselineskip}
 in the language of Section \ref{background}. We obtain the following multiplication table for the tensor products of $\overline{V_5}$ and $\overline{V_7}$ by direct computation (we also include a proof of this below in Section  \ref{sect-char3} using the results of Section \ref{class-maps}): 
\vspace{\baselineskip}
 \begin{center}
 \bgroup
\def\arraystretch{1.5}
\begin{tabular}{ c|c|c| } 
  & $\overline{V_5}$ & $\overline{V_{7}}$\\ \hline
  $\overline{V_5}$ & $\overline{\kk} \oplus\overline{V_5} \oplus \overline{V_{7}}$ & $\overline{V_5}$ \\ \hline
  $\overline{V_7}$ & $\overline{V_5}$ & $\overline{\kk}$ \\ \hline
\end{tabular}
\egroup
\end{center}
\vspace{\baselineskip}
That is, for instance, we have $\overline{V_5} \otimes \overline{V_7} \cong \overline{V_5}$. Since $\kk$ appears as a summand of $V_5 \otimes V_5$ and $V_7 \otimes V_7$, and since both $V_5$ and $V_7$ are indecomposable by Corollary \ref{indecomp-gr}, it follows from \cite[Theorem 2.1]{Benson-Carlson} that $V_5$ and $V_7$ are self-dual (although this is also straightforward to observe by a general description of the duals of $\alpha_p(r,s)$-representations which correspond to skew diagrams, see \cite[Lemma 2.2.6]{cao2023decomposition}). Therefore, the subcategory of $\overline{\rep}(\alpha_3(0,2))$ generated by $\overline{V_5}$ contains three simple objects: $\overline{\kk}, \overline{V_5}$, and $\overline{V_7}$. The objects $\overline{V_7}$ and $\overline{\kk}$ both have dimension 1 (since they are invertible), and $\overline{V_5}$ has dimension 2 (since the square of the dimension of $\overline{V}_5$ is equal to the dimension of $\overline{V}_5$ plus 2). Hence, any group scheme $G'$ such that $\rep(G')$ could be equivalent to this subcategory of $\overline{\rep}(\alpha_3(0,2))$ would necessarily be 6-dimensional. The Grothendieck ring of this subcategory is isomorphic to the Grothendieck ring of $\rep(\mathbb{C}, S_3)$, representations of $S_3$ in characteristic 0, where $S_3$ is the symmetric group on three elements. By \cite{Coulembier-Etingof-Ostrik}, this subcategory must be equivalent to representations of some super-group scheme. In fact, by a super-Nagata Theorem, it must be equivalent to the category of representations for a linearly reductive group scheme, see \cite[Section A.2.3]{EtingofKannan2023}. Therefore, there is only one option: this subcategory is equivalent to $\rep(\mathbb{Z}/2\mathbb{Z} \ltimes \mu_3)$, where $\mu_3$ is the dual group scheme to $\mathbb{Z}/3\mathbb{Z}$. This category is the modulo 3 reduction of $\rep(\mathbb{C}, S_3)$, in the language of \cite{EtingofGelaki2011}. Since there is no group $G'$ with order coprime to $p$ such that $\rep(G')$ is equivalent to this subcategory, the only induced $p'$-symmetry we get from this example is for the group $\mathbb{Z}/2\mathbb{Z}$. Thus this example is in line with Conjecture \ref{etingof-conj-2}. One of the main results in this paper, Theorem \ref{mainthm-p3}, was motivated by the observation that $V_5$ appears as a summand of $V \otimes V^*$ for a large family of representations $V$, giving us a family of induced $\mathbb{Z}/2\mathbb{Z}$ $p'$-symmetries; the proof of this is completed in the next section. 
\end{example}

Although in this paper we will focus primarily on cyclic representations (that is, in the language of diagrams, representations corresponding to diagrams with a single lower left corner), we give one example of a non-cyclic $p'$-representation which is not $p'$-invertible, and look at its possible induced $p'$-symmetries.

\begin{example}
Let $V$ be the $\alpha_3(1,1)$-representation corresponding to the diagram
\vspace{\baselineskip}
\ytableausetup{notabloids}
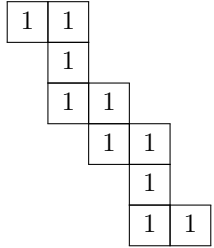
\begin{center}
\begin{ytableau}
1 & 1\\
\none & 1\\
\none & 1& 1 \\
\none & \none & 1 &1\\
\none & \none & \none & 1\\
\none & \none & \none & 1 & 1\\
\end{ytableau}
\end{center}
\vspace{\baselineskip}
Note that $V$ is self-dual, and one can check that $V \otimes V$ decomposes into $\kk, V_{10},$ $V_{16},$ and $V_{34}$, where $V_i$ is dimension $i$. In particular, $V$ is not $p'$-invertible, but each of these summands of $V \otimes V$ is $p'$-invertible. The multiplication table for these four representations, as computed in the semisimplification, is given by:
\vspace{\baselineskip}
\begin{center}
 \bgroup
\def\arraystretch{1.5}
\begin{tabular}{ c|c|c|c|c| } 
  & $\overline{V}$ & $\overline{V_{10}}$ & $\overline{V_{16}}$ & $\overline{V_{34}}$ \\ \hline
  $\overline{V}$ & $\overline{\kk} \oplus \overline{V_{10}} \oplus \overline{V_{16}} \oplus \overline{V_{34}}$ & $\overline{V}$ & $\overline{V}$& $\overline{V}$ \\ \hline
 $\overline{V_{10}}$ &$\overline{V}$ & $\overline{\kk}$ & $ \overline{V_{34}}$ & $\overline{V_{16}}$\\ \hline
 $\overline{V_{16}}$ &$\overline{V}$ &$\overline{V_{34}}$ &$\overline{\kk}$ &$\overline{V_{10}}$ \\ \hline
 $\overline{V_{34}}$&$\overline{V}$ &$\overline{V_{16}}$ &$\overline{V_{10}}$ &$\overline{\kk}$ \\
 \hline
\end{tabular}
\egroup
\end{center}
\vspace{\baselineskip}
We see that all indecomposables are dimension 1 except for $\overline{V},$ which must be dimension $2$ (since $\overline{V} \otimes \overline{V}$ breaks down into a direct sum of four indecomposables, each of which is 1-dimensional). Therefore, if there is a group $G'$ with $\rep(G')$ tensor-equivalent to the subcategory of $\overline{\rep}(G)$ generated by $\overline{V}$, $G'$ would necessarily be order 8. Since one of the simple objects has dimension greater than 1, $G'$ would necessarily be non-abelian, hence either $D_4$ the dihedral group of order 8 or $Q_8$ the quaternion group of order 8. The Grothendieck rings of the representation categories of these two groups are isomorphic, so we cannot eliminate either as a possibility using the multiplication table alone. However, in either case, $G'$ would necessarily have order $8$, consistent with Conjecture \ref{etingof-conj-2}.
\end{example}

\begin{example}
    The $\alpha_3(1,2)$-representation $V$ corresponding to the diagram 
    \vspace{\baselineskip}
\ytableausetup{notabloids}
\begin{center}
\begin{ytableau}
1 \\
1 \\
1 \\
1 \\
1 \\
1 & 1 & 1\\
1 & 1 & 1\\
\end{ytableau}
\end{center}
\vspace{\baselineskip}
is a $p'$-representation which is not $p'$-invertible. Namely, the representation $V \otimes V^*$ decomposes into indecomposables of dimensions 1, 3, 9, 9, 46, and 53. Hence $V \otimes V^*$ has three non-negligible summands. The tensor product $V \otimes V$ decomposes into indecomposables of dimensions 3, 9, 9, 23, 25, and 52, that is, it likewise has three non-negligible summands. It would be interesting to determine the subcategory of $\overline{\rep}(\alpha_3(1,2))$ generated by $\overline{V}$. 
\end{example}

\section{Proof of the main theorems}
\label{proof-mainthms}

In this section, we prove \Cref{maintheorem} and its $p=3$ corollary, \Cref{mainthm-p3}.
\subsection{Classifying maps $V \to V \otimes W$}
\label{class-maps}

Let $V$ be a cyclic representation of $\alpha_p(r,s)$ generated in degree (0,0). Let $W$ be the cyclic representation of $\alpha_p(r,s)$ of dimension $n+m+1$ generated in degree $-n$, where $n$ and $m$ are positive integers, such that $x$ acts by $0$. In this section, we classify the $\alpha_p(r,s)$-representation homomorphisms $V \to V \otimes W$, which will be useful in the following section.

Since $V$ is cyclic, any graded map $f: V \to V \otimes W$ is determined by the collection $(a_0,..., a_n) \in \kk^{n+1}$, where 
\[
f(v_{00})=a_0 v_{0,0} \otimes w_{0,0} + a_1 v_{0,1} \otimes w_{0,-1}+...+a_n v_{0,n} \otimes w_{0,-n},
\]
where we use $v_{i,j}$ and $w_{i,j}$ respectively for the standard basis elements of $V$ and $W$.

\begin{lemma}
    \label{xiyj}
    Let $f$ be as above. Then
    \begin{align*}
    f(v_{ij}) =& a_0 \left ( \sum_{k=0}^j \binom{j}{k} v_{i,k}\otimes w_{0,j-k} \right ) \\
    &+ a_1 \left (\sum_{k=0}^j \binom{j}{k} v_{i,k+1} \otimes w_{0,j-k-1} \right)\\
    &+...\\
    &+ a_n \left (\sum_{k=0}^j \binom{j}{k} v_{i,k+n} \otimes w_{0,j-k-n} \right).
    \end{align*}
\end{lemma}

\begin{proof}
We have
\begin{align*}
    f(v_{i,j})&=f(x^i y^j v_{0,0})\\
    &=x^i y^j f(v_{0,0})\\
    &= x^i y^j (a_0 v_{0,0} \otimes w_{0,0} + a_1 v_{0,1} \otimes w_{0,-1}+...+a_n v_{0,n} \otimes w_{0,-n})\\
    &=(x \otimes 1 + 1 \otimes x)^i (y \otimes 1 + 1 \otimes y)^j (a_0 v_{0,0} \otimes w_{0,0} + a_1 v_{0,1} \otimes w_{0,-1}+...+a_n v_{0,n} \otimes w_{0,-n})\\
    &=(y \otimes 1 + 1 \otimes y)^j (a_0 v_{i,0} \otimes w_{0,0} + a_1 v_{i,1} \otimes w_{0,-1}+...+a_n v_{i,n} \otimes w_{0,-n})\\
    &= \left ( \sum_{k=0}^j \binom{j}{k}y^k \otimes y^{j-k} \right )(a_0 v_{i,0} \otimes w_{0,0} + a_1 v_{i,1} \otimes w_{0,-1}+...+a_n v_{i,n} \otimes w_{0,-n}).\\
\end{align*}
The last line then yields the claimed formula.
\end{proof}

\begin{corollary}
    \label{maps-bijection}
    Graded maps $V \to V \otimes W$ are in bijection with solutions $(a_0,..., a_n)$ in $\kk^{n+1}$ to the system of equations
    \begin{align*}
        0&=a_0 \binom{j}{j-1} + a_1 \binom{j}{j-2} +...+a_n \binom{j}{j-n-1}\\
        &= a_0 \binom{j}{j-2} + a_1 \binom{j}{j-3}+...+a_n \binom{j}{j-n-2}\\
        & ...\\
        &= a_0 \binom{j}{j-\min\{ j,m\}} + a_1 \binom{j}{j-\min\{j,m\}-1}+...+a_n \binom{j}{j-\min\{j,m\}-n}
    \end{align*}
    over all pairs $(i,j)$ where $v_{i,j}=0$ and $v_{i,j-1} \not = 0$.
\end{corollary}

\begin{proof}
    By Lemma \ref{xiyj}, if we have $f: V \to V \otimes W$ with $f(v_{0,0})=a_0 v_{0,0} \otimes w_{0,0} +...+a_n v_{0,n} \otimes w_{0,-n}$, we have a formula for what $f(v_{i,j})$ must be. The only obstruction to defining such a graded map of representations given any choice of $(a_0,...,a_n)$ is that $v_{i,j}$ is eventually equal to 0, for large enough $i$ and $j$; for those $i$ and $j$, one of course will need $f(v_{i,j})=0$. The system of equations then arises by picking out the coefficients in the formula given in Lemma \ref{xiyj} for, in order, $v_{i,j-1} \otimes w_{0,1}, v_{i,j-2} \otimes w_{0,2},..., v_{i, j-\min\{j,m\}} \otimes w_{0,\min\{j,m\}}.$
\end{proof}

We write $w^{i,j}$ and $v^{i,j}$ for the dual basis elements of $W^*$ and $V^*$, respectively, corresponding to standard basis elements $w_{i,j}$ and $v_{i,j}$. In what follows, we use the fact that $W^* \cong W$ as $\alpha_p(r,s)$-representations. Under this isomorphism, $w^{0,m}$ corresponds to $w_{0,-n}$. We assume for this proposition that $m=n$.

\begin{proposition}
\label{summand-formula}
    Assume $m=n$. Let $f$ and $g$ be the maps $V \to V \otimes W$ corresponding to $(a_0,..., a_n)$ and $(b_0,..., b_n) \in \kk^{n+1}$ as in \Cref{maps-bijection}, respectively. Denote by $\hat f$ the map $W \to V \otimes V^*$ and $\tilde{g}$ the map $V \otimes V^* \to W$ the maps corresponding to $f$ and $g$ as in \cite[Proposition 2.10.8]{EGNO2015}. Then $\tilde{g} \circ \hat f \not = 0$ if and only if $b_n \not = 0$, and 
    \begin{align*}
    &\sum_{i,j \text{ such that } v_{i,j} \not = 0} \left ( a_0 \binom{j}{n}+a_1 \binom{j}{n+1}+...+ a_n \binom{j}{2n} \right )\\
    &=    \sum_{j} \operatorname{length}(j)\left ( a_0 \binom{j}{n}+a_1 \binom{j}{n+1}+...+ a_n \binom{j}{2n} \right )
    \not = 0,
    \end{align*}
    where $\operatorname{length}(j)$ is the length of the row at height $j$ and the second sum is over all values $j$ where the diagram of $V$ has a row at height $j$.
\end{proposition}

\begin{proof}
    We have the formulas
    \begin{align*}
    \hat{f} &= (\id_{V^* \otimes V} \otimes \ev_W) \circ (\id_{V^*} \otimes f \otimes \id_W) \circ (\coev_V \otimes \id_W),\\
    \tilde{g} &= (\ev_V \otimes\id_W) \circ(\id_{V^*} \otimes g),
    \end{align*}
where we use implicitly that $W \cong W^*$ and that the tensor product commutes up to natural isomorphism. The composition $\tilde{g} \circ \hat f$ is nonzero if and only if it sends $w_{0,-n}$ to something nonzero. The proof then follows from a direct computation. Using Lemma \ref{xiyj}, the map $\hat{f}$ sends $w_{0,-n}$ to 
\[
\sum_{i,j} \left ( a_0 \binom{j}{j-n} + a_1 \binom{j}{j-n-1}  +... + a_n \binom{j}{j-2n} \right) v_{i,j-n} \otimes v^{i,j},
\]
since we are applying the evaluation map on the $W$ factors (meaning that the terms surviving in $V \otimes V^*$ are those which are tensored with $w_{0,n}$). Applying $\tilde{g}$ to this then yields
\[
b_n \sum_{i,j} \left ( a_0 \binom{j}{j-n} + a_1 \binom{j}{j-n-1}  +... + a_n \binom{j}{j-2n} \right) w_{0,-n}.
\]

\end{proof}

Therefore, to prove that $W$ is a direct summand of $V \otimes V^*$, it suffices to produce solutions $(a_0,...,a_n)$ and $(b_0,...,b_n)$ to the system given in Corollary \ref{maps-bijection} such that $b_n \not = 0$ and $a_0,..., a_n$ satisfy the equation given in Proposition \ref{summand-formula}.

\subsection{Main theorem for general $p$}

In this section, we prove Theorem \ref{maintheorem}. We therefore specialize to the case that $V_h$ is a cyclic representation generated in degree $(0,0)$. Set $m_1:=\lceil \log_p(h) \rceil$, $m_2:=\lceil \log_p(\dim(W)) \rceil$, and $g:=\max(m_1, m_2)$. In particular, we have $p^g\geq h$ and $p^g \geq \dim(W)$. Suppose $V$ is a cyclic representation of dimension coprime to $p$ such that every column of the diagram of $V$ has length equal to 0 or $h$ modulo $p^g$. Let $W$ be a cyclic representation where $x$ acts by 0 generated in degree $-n$. We will do the case in this section that $\dim(W)$ is odd; the case that $\dim(W)$ is even is similar. In particular, set $\dim(W)=2n+1$, so that its highest degree is in degree $n$ (i.e., in the language of the previous section, we take $m=n$). 

We first note an elementary consequence of Lucas's Theorem which we will use throughout this section.

\begin{lemma}
    \label{lucas}
    Suppose $h$ and $j$ are congruent modulo $p^g$. Then for any integer $0\leq l < p^g$, we have $\binom{j}{l} = \binom{h}{l}.$
\end{lemma}

\begin{proof}
Let the base $p$ expansion of $h$ be expressed as 
\[
\beta_{g-1}p^{g-1} + \beta_{g-2}p^{g-2} + \cdots +\beta_1p+\beta_0.
\]
Then, since $j\equiv h\pmod{p^g},$ we can express $j$ as 
    \[
    \beta_ep^e + \beta_{e-1}p^{e-1} + \cdots + \beta_gp^g + (\beta_{g-1}p^{g-1}+\beta_{g-2}p^{g-2} + \cdots +\beta_1p + \beta_0).
    \]
In addition, denote the base $p$ expansion of $l$, where $l < p^g,$ by 
\[
\gamma_{g-1}p^{g-1} + \gamma_{g-2}p^{g-2} + \cdots + \gamma_1p + \gamma_0.
\]
By Lucas's Theorem, it follows that 
\[
\binom{h}{l}\equiv \prod_{i=0}^{g-1} \binom{\beta_i}{\gamma_i} \pmod{p}.
\]
Similarly, we have that
\[
  \binom{j}{l}\equiv \prod_{i=0}^e \binom{\beta_i}{\gamma_i} = \prod_{i=0}^{g-1}\binom{\beta_i}{\gamma_i}\cdot \left(\prod_{i=g}^{e} \binom{\beta_i}{\gamma_i} \right) = \prod_{i=0}^{g-1}\binom{\beta_i}{\gamma_i}\cdot \left(\prod_{i=g}^{e} \binom{\beta_i}{0} \right) = \prod_{i=0}^{g-1}\binom{\beta_i}{\gamma_i}.
\]
\end{proof}

\begin{lemma}
    \label{maps-same}
    A tuple $(a_0,...,a_n)\in \kk^{n+1}$ corresponds to a map $V_h \to V_h \otimes W$ as in \Cref{maps-bijection} if and only if it corresponds to a map $V \to V \otimes W$.
\end{lemma}

\begin{proof}
    Since $(a_0,...,a_n)\in \kk^{n+1}$ corresponds to a map $V_h \to V_h \otimes W$, the system of equations
    \begin{align}
    \label{sys-one}
        0&=a_0 \binom{h}{h-1} + a_1 \binom{h}{h-2} +...+a_n \binom{h}{h-n-1}\\
        &= a_0 \binom{h}{h-2} + a_1 \binom{h}{h-3}+...+a_n \binom{h}{h-n-2}\\
        & ...\\
        &= a_0 \binom{h}{h-\min\{ h,n\}} + a_1 \binom{h}{h-\min\{h,n\}-1}+...+a_n \binom{h}{h-\min\{h,n\}-n}
    \end{align}
    is satisfied. In order that $(a_0,...,a_n)$ correspond to a map $V \to V \otimes W$, we have to verify the equations
    \begin{align}
    \label{sys-two}
        0&=a_0 \binom{j}{j-1} + a_1 \binom{j}{j-2} +...+a_n \binom{j}{j-n-1}\\
        &= a_0 \binom{j}{j-2} + a_1 \binom{j}{j-3}+...+a_n \binom{j}{j-n-2}\\
        & ...\\
        &= a_0 \binom{j}{j-\min\{ j,n\}} + a_1 \binom{j}{j-\min\{j,n\}-1}+...+a_n \binom{j}{j-\min\{j,n\}-n}
    \end{align}
    over all $j$ which satisfies $v_{i,j} = 0$ and $v_{i,j-1} \not = 0$. Since the heights of the columns of $V$ are congruent to either 0 or $h$ modulo $p^g$, we just need to check that the above equations for $j$ congruent to $0$ and $h$ modulo $p^g$ are equivalent to (1)-(4). Note that $V$ must have at least one column of height congruent to $h$ modulo $p^g$, or else $V$ would have dimension divisible by $p$. We claim that if $j$ is divisible by $p^g$, then the system (5)-(8) is trivial (namely, all binomial coefficients are 0) and if $j$ is congruent to $h$ modulo $p^g$, then (5)-(8) is equivalent to (1)-(4). Indeed, these claims follow directly from Lemma \ref{lucas}.
\end{proof}

We assume for the remainder of the section that $W$ is a direct summand of $V_h \otimes V_h^*$. By Corollary \ref{maps-bijection} and Proposition \ref{summand-formula} we have tuples $(a_0,...a_n)$ and $(b_0,...,b_n)$ such that $b_n \neq 0$ and 
\begin{align*}
    c&:=\sum_{j=0}^{h-1} \left ( a_0 \binom{j}{n}+a_1 \binom{j}{n+1}+...+ a_n \binom{j}{2n} \right ) \\
    &= a_0\binom{h}{n+1} + a_1\binom{h}{n+2} + \cdots + a_n\binom{h}{2n+1}
    \end{align*}
satisfies $c \neq 0$ (using the Hockey Stick Identity, see \cite[Theorem 1.2.3]{west2021combinatorial}). Our goal for the remainder of the proof will be to show that
\begin{equation}
\label{thing-showing-nonzero}
\sum_{j} \operatorname{length}(j)\left ( a_0 \binom{j}{n}+a_1 \binom{j}{n+1}+...+ a_n \binom{j}{2n} \right )
    \not = 0,
\end{equation}
where the sum is taken over all rows $j$ of the diagram of $V$; this will imply immediately that $W$ is a direct summand of $V \otimes V^*$.

The strategy for this proof is to break down (\ref{thing-showing-nonzero}) into a sum of expressions of the form 
    \begin{equation}
    \label{breakdown}
    \operatorname{length}(j_1) \sum_{j=j_1}^{j_2-1} \left ( a_0 \binom{j}{n}+a_1 \binom{j}{n+1}+...+ a_n \binom{j}{2n} \right ),
    \end{equation}
    where $j_1$ and $j_2$ are rows such that 
    \[
    \length(j_1 -1) > \length(j_1) = \length (j_1 + 1) =... = \length (j_2 -1) > \length(j_2).
    \]
Since $j_1$ and $j_2$ will be heights of columns, by assumption each is either equivalent to 0 or $h$ modulo $p^g$. We will consider three cases in order to understand (\ref{breakdown}): the case where $j_1$ and $j_2$ are equivalent modulo $p^g$, the case where $j_1\equiv 0$ and $j_2 \equiv h$ modulo $p^g,$ and the case where $j_1 \equiv h$ and $j_2 \equiv 0$ modulo $p^g$. 

\begin{lemma}
    \label{j1-j2-same}
    If $j_1 \equiv j_2$ modulo $p^g$, then 
    \[
    \operatorname{length}(j_1) \sum_{j=j_1}^{j_2-1} \left ( a_0 \binom{j}{n}+a_1 \binom{j}{n+1}+...+ a_n \binom{j}{2n} \right )=0.
    \]
\end{lemma}

\begin{proof}
We compute directly:
       \begin{align*}
        & \sum_{j=j_1}^{j_2-1} \left ( a_0 \binom{j}{n}+a_1 \binom{j}{n+1}+...+ a_n \binom{j}{2n} \right ) \\
        &= \frac{j_2-j_1}{p^g}\cdot \sum_{j=0}^{p^g-1} \left ( a_0 \binom{j}{n}+a_1 \binom{j}{n+1}+...+ a_n \binom{j}{2n} \right ) \\
        &= \frac{j_2-j_1}{p^g}\cdot \left(\sum_{i=0}^n \sum_{j=0}^{p^g-1} a_i\binom {j}{n+i} \right)\\
        &= \frac{j_2-j_1}{p^g}\cdot \left(\sum_{i=0}^n a_i\binom {p^g}{n+i+1} \right) \\
        &=0.
    \end{align*}
To see the final step, note that because $g=\max(\lceil \log_p(h) \rceil, \lceil \log_p(2n+1) \rceil),$ we must have that $g \geq \lceil \log_p(2n+1) \rceil$, and consequently that $g > \log_p(i)$ for $i=n+1, n+2, \ldots, 2n$. Then, by Lemma \ref{lucas} we see that $\binom{p^g}{n+i+1}\equiv 0$ for $i=0,1,\ldots, n,$ and the final step follows. 
\end{proof}

\begin{lemma}
        \label{j1-j2-not-same}
    If $j_1 \equiv 0$ and $j_2 \equiv h$ modulo $p^g$, then 
    \[
    \operatorname{length}(j_1) \sum_{j=j_1}^{j_2-1} \left ( a_0 \binom{j}{n}+a_1 \binom{j}{n+1}+...+ a_n \binom{j}{2n} \right )=c \cdot \length(j_1).
    \]
\end{lemma}

\begin{proof}
    Let $j_1 = p^gc_1$ and $j_2 = p^gc_2 + h.$ Then we have
\begin{align*}
    & \operatorname{length}(j_1) \sum_{j=p^gc_1}^{p^gc_2+h-1} \left ( a_0 \binom{j}{n}+a_1 \binom{j}{n+1}+...+ a_n \binom{j}{2n} \right ) \\
    & =  \operatorname{length}(j_1) \sum_{j=p^gc_1}^{p^gc_2-1} \left ( a_0 \binom{j}{n}+a_1 \binom{j}{n+1}+...+ a_n \binom{j}{2n} \right ) \\
    & + \operatorname{length}(j_1) \sum_{j=p^gc_2}^{p^gc_2+h-1} \left ( a_0 \binom{j}{n}+a_1 \binom{j}{n+1}+...+ a_n \binom{j}{2n} \right ) \\
    &= \operatorname{length}(j_1) \sum_{j=p^gc_2}^{p^gc_2+h-1} \left ( a_0 \binom{j}{n}+a_1 \binom{j}{n+1}+...+ a_n \binom{j}{2n} \right ) \\
    &= \operatorname{length}(j_1) \sum_{j=0}^{h-1} \left ( a_0 \binom{j}{n}+a_1 \binom{j}{n+1}+...+ a_n \binom{j}{2n} \right ) \\
    &= c\cdot \operatorname{length}(j_1).
\end{align*}
\end{proof}

\begin{lemma}
        \label{j1-j2-not-same-two}
    If $j_1 \equiv h$ and $j_2 \equiv 0$ modulo $p^g$, then 
    \[
    \operatorname{length}(j_1) \sum_{j=j_1}^{j_2-1} \left ( a_0 \binom{j}{n}+a_1 \binom{j}{n+1}+...+ a_n \binom{j}{2n} \right )=-c \cdot \length(j_1).
    \]
\end{lemma}

\begin{proof}

Let $j_1=p^gd_1 + h$ and $j_2= p^gd_2$. In this case, we can compute
\begin{align*}
    & \operatorname{length}(j_1) \sum_{j=p^gd_1+h}^{p^gd_2-1} \left ( a_0 \binom{j}{n}+a_1 \binom{j}{n+1}+...+ a_n \binom{j}{2n} \right ) \\
    & =  \operatorname{length}(j_1) \sum_{j=p^gd_1+h}^{p^gd_2+h-1} \left ( a_0 \binom{j}{n}+a_1 \binom{j}{n+1}+...+ a_n \binom{j}{2n} \right ) \\
    & - \operatorname{length}(j_1) \sum_{j=p^gd_2}^{p^gd_2+h-1} \left ( a_0 \binom{j}{n}+a_1 \binom{j}{n+1}+...+ a_n \binom{j}{2n} \right ) \\
    & = - \operatorname{length}(j_1) \sum_{j=p^gd_2}^{p^gd_2+h-1} \left ( a_0 \binom{j}{n}+a_1 \binom{j}{n+1}+...+ a_n \binom{j}{2n} \right ) \\
    & = - \operatorname{length}(j_1) \sum_{j=0}^{h-1} \left ( a_0 \binom{j}{n}+a_1 \binom{j}{n+1}+...+ a_n \binom{j}{2n} \right ) \\
    &= -c\cdot \operatorname{length}(j_1).
\end{align*}
\end{proof}

We set some notation that we will use for the remainder of the proof. 

\begin{notation}
Let $\height(i)$ denote the height of the $i$th column. For row $j,$ let $\col(j)$ denote the unique column for which the height of column $\col(j)$ is equal to $j$ and the height of $\col(j)-1 \not = j$ (if such a column exists). For a column $i$, let $\rowlength(i)$ denote the length of a row $j$ for which $i=\col(j)$ (again, if such a row $j$ exists). We call a column $i$ {\em{descending}} if $j$ with $i=\col(j)$ exists, or, equivalently, if $\height(i-1)>\height(i)$. Furthermore, for a row $j$, set $\floor(j)$ to be the highest row such that $\length(\floor(j))>\length(j)$, and set $\roof(j)$ to be the highest row such that $\length(\roof(j))=\length(j)$. If the floor does not exist, we declare $\floor(j):=-1$.   
\end{notation}

\begin{example}
    To illustrate this notation, suppose $V$ is the representation corresponding to the diagram
    \vspace{\baselineskip}
\begin{center}
\ytableausetup{notabloids}
\begin{ytableau}
1 \\
1 \\
1 \\
1 & 1\\
1 & 1 & 1 & 1\\
1 & 1 & 1 & 1\\
\end{ytableau}
\end{center}
\vspace{\baselineskip}
Then there is no $\col(1)$,  $\col(2)$ would be 2 (since column 2 has height 2, and the previous column has height $\not = 2$ -- recall that the first column is column 0, and the first row is row 0), $\col(3)$ would be 1, there is no $\col(4)$ or $ \col(5)$, and $\col(6)$ would be 0. Hence columns 0, 1, and 2 are descending, but column 3 is not. For the descending columns, we have
\begin{align*}
    \rowlength(0)&=\length(6)=1,\\
    \rowlength(1)&=\length(3) =2,\\
    \rowlength(2)&=\length(2)=4. 
\end{align*}
The floors of rows 0 and 1 would be -1, and the roofs of rows 0 and 1 would be 1. The floor and roof of row 2 is itself. The floor of rows 3, 4, and 5 would be 3, and the roof of rows 3, 4, and 5 would be 5. 
\end{example}

Using this notation, we can now summarize the preceding lemmas.

\begin{corollary}
    \label{formula-c-row}
    We have
\begin{gather}
    \sum_{j} \operatorname{length}(j) \left ( a_0 \binom{j}{n}+a_1 \binom{j}{n+1}+...+ a_n \binom{j}{2n} \right ) \notag \\
    = c\cdot \left(\sum_{i \text{ with } \height(i)\equiv 0, \height(i-1)\equiv h} \operatorname{rowlength}(i) -\sum_{i \text{ with } \height(i)\equiv h, \height(i-1)\equiv 0} \operatorname{rowlength}(i) \right) \notag
\end{gather}
where the first sum is over all rows $j$ and the second sums are over all descending columns $i$ satisfying the given conditions.
\end{corollary}

\begin{proof}
    This follows directly from Lemmas \ref{j1-j2-same}, \ref{j1-j2-not-same}, and \ref{j1-j2-not-same-two}.
\end{proof}

We know that $c$ is nonzero by assumption, so by Corollary \ref{formula-c-row}, in order to show that $W$ is a summand of $V \otimes V^*$, it suffices to show that the term 
\[
\left(\sum_{i \text{ with } \height(i)\equiv 0, \height(i-1)\equiv h} \operatorname{rowlength}(i) -\sum_{i \text{ with } \height(i)\equiv h, \height(i-1)\equiv 0} \operatorname{rowlength}(i) \right)
\]
is nonzero. We now give a useful formula involving this expression.

\begin{lemma}
\label{length-vs-rowlength}
    We have $\sum_j \length(j)$, taken over all rows $j$, is equal to
\[
h \cdot \left(\sum_{i \text{ with } \height(i)\equiv 0, \height(i-1)\equiv h} \operatorname{rowlength}(i)-\sum_{i \text{ with } \height(i)\equiv h, \height(i-1)\equiv 0} \operatorname{rowlength}(i) \right),
\]   
where all congruences mean congruence modulo $p^g$.
\end{lemma}

\begin{proof}
We break up the sum $\sum_j \length(j)$ into three pieces:
\begin{align}
     \sum_j \length(j) &= \left(\sum_{j \text{ with } \roof(j) \equiv \floor(j)} \operatorname{length}(j) \right)\label{eq:length-1}\\  &\;\;\;\;\;\;\;\;\;\;+ \left(\sum_{j \text{ with } \floor(j)\equiv 0, \roof(j)\equiv h} \operatorname{length}(j) \right) \label{eq:length-2}\\ &\;\;\;\;\;\;\;\;\;\;+ \left(\sum_{j \text{ with } \floor(j)\equiv h, \roof(j)\equiv 0} \operatorname{length}(j) \right) \label{eq:length-3}, 
\end{align}
where all congruences are taken modulo $p^g$. Here we are using the fact that $\floor(j)$ and $\roof(j)$ are necessarily the heights of columns, hence each of them is congruent to either $h$ or 0 modulo $p^g$.

Consider each of the three expressions in parentheses. We first note that the term \eqref{eq:length-1} evaluates to $0$, as all the rows appearing in this sum are present in blocks of size divisible by $p^g$ (all rows in these blocks are of the same length).

By a similar analysis, the term \eqref{eq:length-3} can be shown to satisfy
\[
\sum_{j \text{ with } \floor(j)\equiv h, \roof(j)\equiv 0} \operatorname{length}(j)\equiv h \cdot \left( \sum_{i \text{ with } \height(i)\equiv 0, \height(i-1)\equiv h} \operatorname{rowlength}(i) \right)
\]
and the term \eqref{eq:length-2} can be shown to satisfy
\[
\sum_{j \text{ with } \floor(j)\equiv 0, \roof(j)\equiv h} \operatorname{length}(j)\equiv h \cdot \left( \sum_{i \text{ with } \height(i)\equiv h, \height(i-1)\equiv 0} \operatorname{rowlength}(i) \right).
\]
In each case, we simply count the number of rows which share a given length.
\end{proof}

We are now ready to collect our lemmas together.

\begin{proof}[Proof of Theorem \ref{maintheorem}]
By Corollary \ref{formula-c-row}, we just need to show that
\begin{gather}
     c\cdot \left(\sum_{i \text{ with } \height(i)\equiv 0, \height(i-1)\equiv h} \operatorname{rowlength}(i) -\sum_{i \text{ with } \height(i)\equiv h, \height(i-1)\equiv 0} \operatorname{rowlength}(i) \right) \notag
\end{gather}
is nonzero. We know that $c$ is nonzero by assumption, so we just need to check that the expression in the parentheses is nonzero. But by Lemma \ref{length-vs-rowlength}, since $\sum_j\length(j)$, summing over all rows (which is the total dimension of $V$), is nonzero (using the assumption that $V$ is a $p'$-representation), it must indeed be the case that 
\[
\left(\sum_{i \text{ with } \height(i)\equiv 0, \height(i-1)\equiv h} \operatorname{rowlength}(i) -\sum_{i \text{ with } \height(i)\equiv h, \height(i-1)\equiv 0} \operatorname{rowlength}(i) \right)
\]
is also nonzero. This completes the proof for the case that $W$ is odd-dimensional; in the case that $W$ has dimension $2n+2$, the proof is similar. In that case, we still have $W$ generated in degree $(0,-n)$, but now its top degree is $(0, n+1)$ instead of $(0,n)$. 

\end{proof}

\subsection{Examples for $p=3$}
\label{sect-char3}

In this section, we now deduce Theorem \ref{mainthm-p3} for $p=3$ from the general result Theorem \ref{maintheorem} which was proved in the previous section.

We recall the terminology as in Example \ref{p3-v5-ex}. We have cyclic representations $\kk,$ $V_3,$ $V_5,$ $V_7,$ and $V_9$ for $\alpha_3(0,2)$ corresponding to column graded diagrams.

\begin{proof}[Proof of Theorem \ref{mainthm-p3}]
We can show that $V_5\otimes V_5^{*}\cong \kk\oplus V_3\oplus V_5\oplus V_7\oplus V_9$ using Corollary \ref{maps-bijection} and Proposition \ref{summand-formula}. For instance, to produce a map $V_5$ to $V_5\otimes V_5^*,$ we can examine the system of equivalences produced by Corollary \ref{maps-bijection}. We get that 
\begin{gather}
2a_0+a_1+a_2= 0, \notag \\
a_0 + a_1+2a_2= 0, \notag
\end{gather}
which simply implies that $(a_0, a_1, a_2)$ must be of the form $(a,0,a)$ for some $a \in \kk$. By Proposition \ref{summand-formula}, in order to check such a map corresponds to a direct summand embedding of $V_5$ in $V_5 \otimes V_5^*$, we just need to show that there is some $(a_0, a_1, a_2) = (a,0,a)$ such that 
\[\sum_{j=0}^{4} \operatorname{length}(j) \left(a_0\binom{j}{2} + a_1\binom{j}{3} + a_2\binom{j}{4} \right)\neq 0.\]
Substituting in the form of $(a_0, a_1, a_2),$ we get that the expression simplifies to 
\[
a\cdot \sum_{j=0}^4 \operatorname{length}(j)\cdot \left( \binom{j}{2} + \binom{j}{4} \right).
\]
Since $\operatorname{length}(j)=1$ for $j=0,1,2,3,4,$ we get that this expression is equal to 
\[
a\cdot \left (\binom{5}{3} + \binom{5}{5} \right) = 2a,
\]
and thus setting $a$ equal to any nonzero value produces the desired nonzero composition of maps, showing that $V_5$ is a summand of $V_5\otimes V_5^*$. Similar computations can be done to determine the other summands of $V_5 \otimes V_5^*$. Since $\overline{V_3} \cong 0 \cong \overline{V_9}$, the decomposition in the semisimplification is given by $\overline{V_5} \otimes \overline{V_5}^* \cong \overline{\kk} \oplus \overline{V_5} \oplus \overline{V_7}$. 

Let $V$ be a cyclic representation of $\alpha_3(r,s)$ generated in degree $(0,0)$ with dimension coprime to $3$ corresponding to a graded diagram with heights of $a_1, a_2, \ldots, a_n$ such that $a_i\equiv 0,5\pmod{9}$ for $i=1,2,\ldots, n$. For all such $V,$ every summand of $V_5\otimes V_5^{*}$ is in the decomposition of $V\otimes V^*$ by Theorem \ref{maintheorem}. In particular, $V_5$ and $V_7$ are in the decomposition of $V\otimes V^*$, and so $V$ is a $p'$-representation which is not $p'$-invertible, and the tensor subcategory of $\overline{\rep}(\alpha_3(r,s))$ contains the semisimple tensor category $\mc{C}$ described in Example \ref{p3-v5-ex} which is equivalent to $\rep(\mathbb{Z}/2\mathbb{Z} \ltimes \mu_3),$ which is the modulo 3 reduction of (hence having the same Grothendieck ring as) $\rep(\mathbb{C},S_3)$.  

If instead $V$ satisfies that every row has length equal to 0 or 5 modulo 9, the result follows by reversing the roles of rows and columns.
\end{proof}

\bibliography{refs}
\bibliographystyle{myamsalpha}

\end{document}